\documentclass{amsart}
\usepackage{amsfonts}
\usepackage{amsmath}
\usepackage{amsthm}
\usepackage{amssymb}
\usepackage[all]{xy}
\usepackage{color}
\usepackage{soul}
\usepackage{bbding}
\usepackage{mathrsfs}
\usepackage{colortbl}
\usepackage{caption}
\usepackage{subcaption}
\usepackage[pdftex,colorlinks,citecolor=blue]{hyperref}
\usepackage{cleveref}

\usepackage{tikz}
\usetikzlibrary{arrows.meta, decorations.pathmorphing, backgrounds, positioning, fit, petri, patterns, decorations.markings}

\definecolor{job}{RGB}{200,65,0}
\definecolor{karin}{RGB}{0,100,200}
\definecolor{notJorK}{RGB}{0,150,0}

\DeclareMathOperator{\Hom}{Hom}
\DeclareMathOperator{\End}{End}

\newcommand{\C}{\mathcal{C}}
\newcommand{\N}{\mathbb{N}}
\newcommand{\kk}{\Bbbk}
\newcommand{\Rad}{\mathsf{Rad}}
\newcommand{\J}{\mathsf{J}}
\newcommand{\Jay}{\mathcal{J}}
\newcommand{\tJay}{\widetilde{\Jay}}
\newcommand{\I}{\mathcal{I}}
\newcommand{\Mod}{\mathsf{Mod}}
\newcommand{\Rep}{\mathsf{Rep}}
\newcommand{\Vc}{\mathsf{Vec}}
\newcommand{\vc}{\mathsf{vec}}
\newcommand{\Q}{\mathcal{Q}}
\newcommand{\Z}{\mathbb{Z}}
\newcommand{\R}{\mathbb{R}}
\newcommand{\Mor}{\mathsf{Mor}}
\newcommand{\Ob}{\mathsf{Ob}}
\DeclareMathOperator{\dist}{\text{d}}

\newtheorem{theorem}{Theorem}[section]
\newtheorem{proposition}[theorem]{Proposition}
\newtheorem{lemma}[theorem]{Lemma}
\theoremstyle{definition}
\newtheorem{definition}[theorem]{Definition}
\newtheorem{remark}[theorem]{Remark}
\newtheorem{example}[theorem]{Example}

\title{Admissible ideals for $\Bbbk$-linear categories}
\author{Karin M.~Jacobsen and Job D.~Rock}
\date{\today}

\begin{document}

\maketitle

\begin{abstract}
    We generalize the notion of an admissible ideal from path algebras to (small) $\kk$-linear categories that satisfy the Krull--Remak--Schmidt--Azumaya assumption.
    In our treatment we first prove some general results that are analogous to general results for path algebras and admissible ideals.
    We then cover generalizations of relations generated by paths of length two, which we call point relations, and more general length relations.
    We conclude the paper with several examples and an appendix containing further discussion on length relations.
\end{abstract}

\section{Introduction}
\label{sec:introduction}
Quivers with relations play a fundamental role in the representation theory of finite dimensional algebras.

Every basic finite dimensional algebra $\Lambda$ over an algebraically closed field $\Bbbk$ is isomorphic to a path algebra $\Bbbk Q/R$ of a finite quiver $Q$ with admissible relations $R$. Moreover, the modules of $\Lambda$ are in bijection with the representations of $(Q,I)$ over $\Bbbk$. Without relations, we get a correspondence only to hereditary algebras \cite{Gabriel}; see also \cite[Ch.~III.1]{ARS}. 

Non-hereditary algebras are central in most fields of modern representation theory of algebras. For one, higher homological algebra requires algebras of global dimension at least 2 \cite{Jasso, Kvamme}. There is a rich tradition of studying classes of non-hereditary algebras, such as gentle \cite{AH, AS}, clannish \cite{C-B}, Schur \cite{Erdmann}, preprojective \cite{Ringel}, and self--injective \cite{SY} algebras.

Continuous quivers and their representations were first explicitly studied in \cite{IRT}. They are a natural generalisation of quivers, replacing finite sets of vertices with uncountably infinite sets. In the process, one gains intuition about what characteristics of representation theory come from innate properties of algebraic structures, and what comes from the discrete examples that are usually studied.

One parameter persistence modules are often defined over the real line so that persistence modules coincide with pointwise finite-dimensional representations of a continuous quiver of type $\mathbb{A}$ (see, for example, \cite{CdSGO}).
In \cite{BBOS} the authors consider $m\times n$ rectangular grid quivers which have the commutativity relation on each square.
The authors of \cite{BBH} study homological approximations in order to obtain new invariants of these representations (persistence modules).

Given the important role of quiver relations in the representation theory of finite-dimensional algebra, it is natural to ask if relations can be extended to the continuous setting. This has already been done in a restricted sense by the second author and Zhu in \cite{RZ}. We give a more general definition that works with any underlying quiver. To capture the full generality we actually go beyond quivers and consider categories instead.

In starting this work, we were motivated by two areas of study that we intend to lift to the continuous setting: gentle algebras and $d$-cluster-tilting subcategories. In gentle algebras, the relations appear in the definition, and are always generated by compositions of two arrows. This type of relations are generalized as \emph{point relations} in \Cref{subsec:point relations}. An important class of $d$-cluster-tilting subcategories appear in the module category of type A algebras, with relations consisting of all paths above a certain length \cite{Vaso}. This type of relations is generalized as \emph{length relations} in \Cref{subsec:length}.

\subsection{Contributions}
In \Cref{sec:definition and general}, we give essential background, before stating our main definition.

\begin{definition}[\Cref{def:admissible}]
Let $\C$ be a category and $\I$ an ideal in $\C$.
    We say $\I$ is \emph{admissible} if the following are satisfied.
    \begin{enumerate}
        \item For each $f$ in $\I$, there exists a finite collection of morphisms $g_1,\ldots,g_n$ not in $\I$ such that $f = g_n\circ\cdots\circ g_1$.
        \item For each nonzero endomorphism $f$, if $f$ is not an isomorphism then there exists $n\geq 2\in\N$ such that $f^n\in\I(X,X)$.
        \end{enumerate}
\end{definition}

\Cref{sec:general results} gives some general results on the quotient category $\C / \I$, summarized here.

\begin{theorem}[\Cref{prop:connected,prop:basic,prop:radical commutes with admissible}]
    Let $\C$ be a  category. Let $\I$ be an admissible ideal of morphisms.
    
    \begin{enumerate}
        \item If $\C$ is connected, then $\C / \I$ is also connected.
        \item If $\C$ is Krull--Remak--Schmidt--Azumaya, then $\C / \mathcal I$ is also Krull--Remak--Schmidt--Azumaya.
        \item If all endomorphism rings of $\C$ are artinian, then \[\Rad(\C / \I)=\Rad(\C) / \I.\]
    \end{enumerate}
\end{theorem}

In \Cref{sec:relations} we give two important classes of relations that can generate admissible ideals. The first is \emph{point relations}, which generalizes relations of length two. The idea is that certain paths through a vertex in the quiver are excluded, but others not. For an illustration see \Cref{fig:point intro}.

\begin{figure}[h!]
    \centering
    \begin{subfigure}[b]{0.3\textwidth}
    \centering
    \begin{tikzpicture}
        \node (center) at (0,0) {$\bullet$};
        \node at (0,.3) {$P$};
        \node (left) at (-1,0) {$\circ$};
        \node (right) at (1,0) {$\circ$};
        \draw[->] (left)--(center);
        \draw[->] (center) -- (right);
    \end{tikzpicture}
    \caption{Discrete case}
    \end{subfigure}
    \begin{subfigure}[b]{0.3\textwidth}
    \centering
    \begin{tikzpicture}[decoration={markings, mark=at position 0.5 with {\arrow{>}}}]
        \node (center) at (0,0) {$\bullet$};
        \node (left) at (-1,0) {};
        \node (right) at (1,0) {};
        \node at (0,.3) {$P$};
        \draw[very thick, postaction=decorate] (left.center) to[bend left] (center.center);
        \draw[very thick, postaction=decorate] (center.center) to[bend right] (right.center);
    \end{tikzpicture}
    \caption{Continuous case}
    \end{subfigure}
    \caption{An illustration of point relations in the discrete and continuous case. In both cases, the relations contain all paths passing through the point $P$.
    See \Cref{fig:arrow and line} on page~\pageref{fig:arrow and line} for an explanation of our drawing conventions.}
    \label{fig:point intro}
\end{figure}
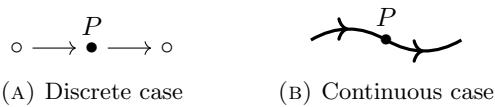

\begin{theorem}[\Cref{thm:point relations are admissible}]
    Let $\{\mathcal P_\alpha\}$ be an admissible collection of point relations in $\C$, such that any cycles are either isomorphisms (and hence trivial) or contained in  at least one $\mathcal P_\alpha$.
    Then $\I=\langle \bigcup_\alpha \mathcal P_\alpha \rangle$ is an admissible ideal in $\C$.
\end{theorem}

The other class of relations we define are \emph{length relations}. This is a generalisation of relations generated by paths containing at least $n$ arrows, where $n$ is a natural number.

\begin{theorem}[\Cref{thm:length relations are admissible}]
    A length relation generates an admissible ideal.
\end{theorem}
\Cref{sec:examples} contains multiple examples of how relations work, including a sketch of their Auslander--Reiten theory.

\subsection{Future Work}
The present paper is a precursor to future work on generalizations of non-hereditary structures.
Of note, the authors will consider point relations, such as \Cref{ex:crossing real lines}, that generalize gentle algebras. They will also study the modding out by length relations, such as \Cref{ex:length relations}(\ref{ex:length relations:continuous A}), to generate higher cluster tilting subcategories.

\subsection{Acknowledgements}
The idea for this project was conceived at the Hausdorff Research Institute of Mathematics, KMJ visited JDR at Ghent University during this project, and JDR visited KMJ at Aarhus University during this project.
The authors thank each of these institutions for their hospitality.
KMJ is supported by the Norwegian Research Council via the project Higher Homological Algebra and Tilting Theory (301046).
JDR is supported at Ghent University by BOF grant 01P12621.
The authors would like to thank Jenny August, Raphael Bennett-Tennenhaus, Charles Paquette, Amit Shah, Emine Y{\i}ld{\i}r{\i}m, and Shijie Zhu for helpful discussions.

\subsection{Conventions}
We work over $\kk=\overline{\kk}$ be a field of characteristic 0.
By $\Vc(\kk)$ and $\vc(\kk)$ we denote the categories of $\kk$-vector spaces and finite-dimensional $\kk$-vector spaces, respectively.
For a $\kk$-algebra $\Lambda$, denote by $\J(\Lambda)$ the Jacobson radical of $\Lambda$.
We assume $\C$ is a $\kk$-linear category.

Recall that a category $\C$ is called \emph{Krull--Remak--Schmidt--Azumaya} if any object $X$ is isomorphic to a arbitrary sum $\bigoplus X_i$, where each $\End_\C X_i$ is a local ring, which itself is unique up to isomorphism.
If every object is instead isomorphic to a finite sum $\bigoplus_{i=1}^n X_i$ as above, we say $\C$ is \emph{Krull--Remak--Schmidt}.

Finally, we consider (discrete) quivers, continuous generalizations of such quivers, and combinations of the two.
When we draw an arrow, we  use a thin line with an arrow head at the end to indicate the direction.
When we draw a continuous line segment, we use a bold line with the arrow head in the middle to indicate the direction;
see \Cref{fig:arrow and line}.

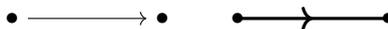
\begin{figure}[h]
    \centering
    \begin{tikzpicture}[decoration={markings, mark=at position 0.5 with {\arrow{>}}}]
        \node (left-start) at (0,0) {$\bullet$};
        \node (left-end) at (2,0) {$\bullet$};
        \node (right-start) at (3,0) {$\bullet$};
        \node (right-end) at (5,0) {$\bullet$};
        \draw[->] (left-start)--(left-end);
        
        \draw[very thick, postaction=decorate] (right-start.center)--(right-end.center);
        
    \end{tikzpicture}
    \caption{On the left, how we draw arrows. On the right, how we draw line segments.}
    \label{fig:arrow and line}
\end{figure}

\section{Definition and General Results}\label{sec:definition and general}

\subsection{$\kk$-linear categorization}

\begin{definition}\label{def:categorification}
    Let $Q$ be a (finite) quiver and $\kk Q$ its path algebra.
    Let $\Q$ be the category whose indecomposable objects are the vertices of $Q$ and morphisms between indecomposables $i$ and $j$ are given by
    \begin{displaymath}
        \Hom_{\Q}(i,j) = e_j \kk Q e_i.
    \end{displaymath}
    The objects in $\Q$ are finite direct sums of the indecomposables (and 0).
    The morphisms in $\Q$ are given by extending bilinearly.
    We call $\Q$ the \emph{$\kk$-linear categorification} of $Q$.
\end{definition}

\begin{example}\label{ex:discrete categorification}
Let $Q$ be the following quiver:

\begin{center}
\begin{tikzpicture}[xscale=1.5, yscale=.5]
    \node (1) at (0,1) {1};
    \node (2) at (1,2) {2};
    \node (3) at (1,0) {3};
    \node (4) at (2,1) {4};
    
    \draw[->] (1) -- node[pos=0.6, above]{$\alpha_1$} (2);
    \draw[->] (1) -- node[pos=0.6, below]{$\beta_1$} (3);
    \draw[->] (2) -- node[pos=0.4, above]{$\alpha_2$} (4);
    \draw[->] (3) -- node[pos=0.4, below]{$\beta_2$} (4);
\end{tikzpicture}
\end{center}
Then the $\kk$-linear categorification $\Q$ is a category with indecomposable objects $1, 2, 3$ and $4$. The morphisms in $\Q$ are given by paths in $Q$, so for example we have $\Hom(1,4) \cong\kk^2$, while  $\Hom(4,1) = 0$
\end{example}

\begin{proposition}\label{lem:correspondence between paths and morphisms}
    There is a bijection between nonzero elements in $\kk Q$ and nonzero morphisms in $\Q$.
\end{proposition}
\begin{proof}
    A non-zero element in $\kk Q$ is a finite sum of paths in $Q$. We can therefore define a map $F$ from the elements in $\kk Q$ to morphisms in $\Q$ by specifying the action of $F$ on paths in $Q$. We let this mapping be determined by $\Hom_{\Q}(i,j) = e_j \kk Q e_i$. This map is a bijection by bilinearity of $\Q$.
\end{proof}

\begin{lemma}\label{prop:same representations}
    Let $\Mod(\Q)$ be the category of functors $\Q\to \Vc(\kk)$.
    Then there exists an isomorphism of categories $\Phi:\Mod(\Q) \to \Rep(Q)$.
\end{lemma}

\begin{proof}
    Let $F$ be a functor in $\Mod(\Q)$.
    We now define the corresponding representation $V=\Phi(F)$.
    Let $M$ be the representation of $Q$ over $\kk$ where $V(i)=F(i)$ for each $i\in Q_0$.
    For a path $\rho$ in $Q$, let $V(\rho)$ be the $\kk$-linear map $F(\rho)$.
    
    Let $f:F\to G$ be a morphism in $\Mod(\Q)$.
    Then $\Phi(f):\Phi(F)\to \Phi(G)$ is defined by the $f_i:F(i)\to G(i)$ for each $i\in Q_0$.
    Straightforward computations show that $\Phi$ respects composition and so it is a functor.
    
    Define $\Phi^{-1}:\Rep(Q)\to \Mod(\Q)$ in the following way.
    For a representation $V$ of $Q$, let $F=\Phi^{-1}(V)$ be determined by $F(i)=V(i)$, for each $i\in Q_0$, and $F(\rho)=V(\rho)$ for each path in $Q$.
    Morphisms are defined similarly.
    One may check $\Phi^{-1}\Phi$ and $\Phi\Phi^{-1}$ are the identity functors on $\Mod(\Q)$ and $\Rep(Q)$, respectively.
\end{proof}

\subsection{The Jacobson Radical}\label{sec:jacobson radical}

\begin{definition}\label{def:radical of C}
    Let $\C$ be a category.
    The \emph{radical} $\Rad(\C)$ of $\C$ is the ideal consisting of
    \begin{displaymath}
        \Rad_{\C}(X,Y) := \left\{ f\in\Hom_{\C}(X,Y) \mid \forall g\in\Hom_{\C}(Y,X),\, f\circ g\in\J(\End_{\C}(Y))  \right\},
    \end{displaymath}
    for each pair of objects $X,Y$ in $\C$.
\end{definition}

\begin{proposition}[\cite{Krause}]\label{prop:Krause}
    Let $X$ and $Y$ be objects in $\C$.
    Then $\Rad_{\C}(X,Y)=\J(\Hom_{\C}(X,Y))$.
\end{proposition}

\begin{proposition}\label{prop:easy radical}
    Let $f:X\to Y$ be an morphism for indecomposable objects $X,Y$ in $\C$.
    Then $f\in\Rad(\C)$ if and only if $f$ is not an isomorphism.
\end{proposition}
\begin{proof}
    Suppose $f$ is not an isomorphism.
    If $\C$ does not have cycles we are done.
    If $\C$ has cycles, let $g:Y\to X$ be a nonzero morphism.
    Then $f\circ g\in\J(\End_{\C}(Y))$ and so $f\in \Rad_{\C}(X,Y)$.
    Reversing the argument shows that if $f\in\Rad_{\C}(X,Y)$ then $f$ is not an isomorphism.
\end{proof}

Recall that $\C$ is semi-simple if every object in $\C$ is a finite direct sum of simple objects and all such direct sums exist.

\begin{proposition}\label{prop:C is basic}
    If $\C$ is Krull--Remak--Schmidt, then $\C/ \Rad(\C)$ is semi-simple.
\end{proposition}
\begin{proof}
    Let $\C$ be a Krull--Remak--Schmidt category.
    Let $X$ and $Y$ be indecomposables in $\C$ such that $X\not\cong Y$.
    Then $\Hom_{\C}(X,Y)=\Rad_{\C}(X,Y)$ and so $\Hom_{\C / \Rad(\C)}(X,Y)=0$.
    Extending bilinearly we see $\C / \Rad(\C)$ is semi-simple.
\end{proof}

\begin{remark}\label{rmk:categorification is basic}
    It follows immediately from \Cref{prop:C is basic} that if $Q$ is a finite acyclic quiver then $\Q /\Rad(\Q)$ is semi-simple.
\end{remark}

\subsection{Admissible Ideals}\label{sec:admissible ideals}

\begin{definition}[\cite{Krause}]
    Let $\{\C_i\}_{i\in I}$ be a family of full additive subcategories of $\C$. We have an \emph{orthogonal decomposition} $\coprod_{i\in I} \C_i$ of $\C$ if every object $X$ in $\C$ is isomorphic to a direct sum $\bigoplus_{i\in I} X_i$, where $X_i$ is an object of $\C_i$, and for $X_i\in \C_i,X_j\in C_j$ we have $\Hom_{\C}(X_i,X_j)=0$ when $i\neq j$.
    
    We say $\C$ is \emph{connected} if the only orthogonal decomposition of $\C$ is the trivial one. 
\end{definition}

An \emph{ideal} $\I$ of a category $C$ is a collection of morphisms is $\C$ such that for any $f\in \I$ and for any $g$ and $h$ such that the composition $gfh $ is defined, the composition $gfh\in \I$.
For an ideal $\I$ of $\C$, we denote by $\I(X,Y)$ the morphisms in $\Hom_{\C}(X,Y)\cap \I$.

\begin{remark}
    For an ideal $\I$ of $\C$, the category $\C /\I$ has the same objects as $\C$.
    The morphisms of $\C / \I$ are given by $\Hom_{\C}(X,Y) / \I(X,Y)$.
    A representation $V: \C / \I \to \kk\text{vec}$ is also a representation of $\C$ by precomposition with the quotient functor $\pi$. Thus we obtain a representation $\widetilde{V}:\C \stackrel{\pi}{\to} \C / \I \stackrel{V}{\to} \kk\text{vec}$.
    Hence 
    we may consider the representations of $\C/\I$ as a subcategory of the representations of $\C$.
    In particular, representations of $\C/\I$ are those representations $V$ of $\C$ such that if $f\in\I$ then $V(f)=0$.
\end{remark}

\begin{definition}\label{def:admissible}
    Let $\C$ be a $\Bbbk$-linear, Krull--Remak--Schmidt--Azumaya category and $\I$ an ideal in $\C$.
    We say $\I$ is \emph{admissible} if the following are satisfied.
    \begin{enumerate}
        \item For each $f$ in $\I$, there exists a finite collection of morphisms $g_1,\ldots,g_n$ not in $\I$ such that $f = g_n\circ\cdots\circ g_1$.
        \item For each indecomposable $X$ in $\C$, the endomorphism ring $\End_{\C}(X) / \I(X,X)$ is finite-dimensional.
    \end{enumerate}
\end{definition}

We remark that, in \Cref{def:admissible}(2), we do not want to require that there is some $n$ that works for all nonisomorphism endomorphisms $f$.
See \Cref{ex: cycles length} for an explicit example why.

\begin{lemma}\label{lem: admissible in radical}
	Let $\C$ be a Krull--Remak--Schmidt category and $\I$ an ideal in $\C$.
    If $\I$ is admissible, then $\I$ is contained in the radical of $\C$.
\end{lemma}

\begin{proof}
	Let $f:X\rightarrow Y$ be a morphism in $\I$ between indecomposable objects. Then we know by \Cref{prop:easy radical} that if $f$ is not contained in the radical, it is an isomorphism. However, if $f$ is an isomorphism, we have $1_X\in\I$.
	Then \emph{every} morphism to / from $X$ is in $\I$.
	Thus, if $1_x=g_n\circ\cdots\circ g_1$ for any composition, both $g_n,g_1\in \I$, which contradicts condition \Cref{def:admissible}(1).
	Hence $\I(X,Y)\subseteq \Rad(X,Y)$ for indecomposable $X,Y$. 
	
	Now let $X=\bigoplus_{i=0}^m X_i$ and $Y=\bigoplus_{j=0}^n Y_j$, where each $X_i$ and $Y_j$ is indecomposable. Consider $f\in \I(X,Y)$. We can rewrite $f$ as $f=(f_{ij})$, where $f_{ij}:X_i\rightarrow Y_j$. By composition with the canonical injections and projections, we see that $f_{ij}\in \I(X_i,Y_j)$, so by the above, $f_{ij}\in \Rad(X_i,Y_j)$.
	Then by linearity, $f\in \Rad(X,Y)$.
\end{proof}

Let $Q$ be a finite quiver and let $\Q$ be its $\kk$-linear categorification. Suppose $I$ is an ideal of the path algebra $\kk Q$. We show how to build an ideal $\I$ in $\Q$ from $I$. 

From the definition of the $\kk$-linear categorification, we know that each path in $I$ corresponds to a non-zero morphism in $\Q$, see \Cref{lem:correspondence between paths and morphisms}. We (na\"ively) let $\I$ be the set of morphisms obtained by mapping $I$ to $\Q$. We now show that $\I$ is an ideal of the category $\Q$. 

By $\kk$-linearity of $Q$, it is enough to consider morphisms between indecomposable objects.
Suppose $f\in\I(i,j)$ for some $i,j\in Q_0=\operatorname{Ind}\Q$, and let $g:j\rightarrow k$ and $h:l\rightarrow i$ be two nonzero morphisms in $\Q$.
By \Cref{lem:correspondence between paths and morphisms} we know that $f$ corresponds to an element $\rho$ in $e_j\kk Qe_i$.
Further, $g$ corresponds to an element $\psi$ in $\kk Qe_j$ and $h$ corresponds to an element $\phi$ in $e_i\kk Q$.
Each of $\rho$, $\phi$, and $\psi$ are are sums of paths in $Q$ from the respective source and to the respective target.
Without loss of generality, due to $\kk$-linearity, suppose each of $\rho$, $\phi$, and $\psi$ is a path in $Q$.
We see $\psi\rho\phi$ is an element of $I$ since $I$ is a two-sided ideal containing $\rho$.
The image of the composition $\psi\rho\phi$ is the composition $gfh$, which must therefore be in $\I$.

\begin{proposition}\label{prop:admissible is correct}
    Let $Q$ be a finite quiver and $I$ an ideal of $\kk Q$ as a path algebra.
    Let $\Q$ be the $\kk$-linear category induced by $Q$ and let $\I$ be the ideal induced by $I$ in $\Q$.
    Then $\I$ is an admissible ideal of $\Q$ as in \Cref{def:admissible} if and only if $I$ is an admissible ideal of $\kk Q$.
\end{proposition}

\begin{proof}
    Let $I$ be an admissible ideal of $\kk Q$ and $\rho\in I$.
    We first prove $\mathcal I$ satisfies property (1) of \Cref{def:admissible}.
    Without loss of generality, assume $\rho$ is a path in $Q$.
    Then $\rho = \alpha_n \alpha_{n-1}\cdots \alpha_2\alpha_1$, where each $\alpha_i$ is an arrow in $Q$.
    Now, let $f$ be the morphism in $\Q$ corresponding to $\rho$ and $g_i$ the morphism in $\Q$ corresponding to $\alpha_i$, for each $i$.
    Then we know each $g_i\notin \mathcal I$ and have satisfied property (1) of \Cref{def:admissible}.
    Reversing the argument proves the converse.
    
    If $Q$ has no cycles then the proposition immediately holds for property (2) of \Cref{def:admissible}.
    So, suppose $Q$ has at least one oriented cycle.
    Since $I \subset \Rad^n(\kk Q)$, for some $n\geq 2$, we see that $\mathcal I$ must immediately satisfy property (2) of \Cref{def:admissible}.
    
    Now suppose $\I$ satisfies \Cref{def:admissible}(2).
    Since $Q$ is finite, there are finitely many cycles.
    For each cycle $\rho$ at each vertex $i$, let 
    \[ m_\rho= \min_m \{ m \mid \rho^m \in \I(i,i)\}.\]
    We know such an $m_\rho$ exists since $\End_{\Q}(i) / \I(i,i)$ is finite-dimensional.
    Let $n_\rho$ be $m_\rho$ time the length of $\rho$.
    Then let \[ N = \max\left(\max_\rho \{n_\rho\}\cup \{\text{length of longest path without cycles in }Q\}\right).\]
    Thus, $\Rad^N(\Bbbk Q)\supset I$.
    This concludes the proof.
\end{proof}

\begin{remark}
The second half of the proof above can be extended to more general quivers.
Suppose $\Q$ is the $\kk$-linear categorification of a (not necessarily finite) quiver $Q$ with finitely many cycles. Suppose that for each cycle $\rho$ in the quiver with corresponding morphism $f_\rho:X\rightarrow X$, there is some $n\geq 2$ such that $f_\rho\in \I(X,X)$.
Then $\I$ satisfies criterion (2) in \Cref{def:admissible}. 

For the majority of our examples, this will be the criterion we actually use.
\end{remark}

\begin{example}\label{ex:discrete ideal}
    Consider the quiver from \Cref{ex:discrete categorification}. Let $I$ be the commutative ideal generated by $\{\alpha_2\alpha_1-\beta_2\beta_1\}$.
    
    In the $\kk$-linear categorification, the relation $\alpha_2\alpha_1-\beta_2\beta_1$ can be written as $\left [ \begin{smallmatrix}\alpha_2 & \beta_2\end{smallmatrix}\right]\left [ \begin{smallmatrix}\alpha_1 \\ -\beta_1\end{smallmatrix}\right]$. The ideal generated by this morphism fulfills the criteria for being an admissible ideal. 
\end{example}

\subsection{General Results}\label{sec:general results}

\begin{proposition}\label{prop:connected}
    Let $\C$ be connected, and let $\I$ be an admissible ideal of morphisms. Then $\C / \I$ is connected.
\end{proposition}

\begin{proof}
 Assume towards a contradiction that $\C / \I$ is not connected; then there exists a decomposition of $\C / \I$ into mutually orthogonal subcategories $\C'_1, \cdots, \C'_n$. We can lift these subcategories to subcategories $\C_1, \cdots, \C_n$ of $\C$. As $\C$ is connected, these subcategories cannot all be mutually orthogonal, so assume that there exists some morphism $f:X_i\rightarrow X_j$, with $X_i\in \C_i, X_j\in \C_j$ and $i\neq j$. To preserve mutual orthogonality in $\C/\I$, we must have $f\in \I$. Then since $I$ is an admissible ideal, we can write $f=g_m\circ \cdots \circ g_1$, where each of the $g_1, \cdots, g_m$ are not in $\I$.
 
 Consider $g_1: X_i\rightarrow Y$ and denote its image in $\C / \I$ by $\overline{g_1}$. As $\C / \I$ is not connected, we can write $Y=\bigoplus_{i=1}^n Y_i$, with $Y_i\in \C'_i$ and $\overline{g_1}=(g^1_1,\cdots g_1^n)$. Now, as $\overline{g_1}$ is nonzero, $g_1^k$ is nonzero for some $k$. If $k\neq i$, we have reached a contradiction. If $k=i$, we can repeat the argument with $\overline{g_2}|_{Y_i}$, eventually reaching a contradiction. 
\end{proof}

\begin{proposition}\label{prop:basic}
    Let $\C$ be Krull--Remak--Schmidt--Azumaya and $\I$ an admissible ideal.
    Then $\C / \mathcal I$ is Krull--Remak--Schmidt--Azumaya.
\end{proposition}
\begin{proof}
Let $X$ be an object in $\C / \mathcal I$. Then $X$ is an object in $\C$, which we assume to be Krull--Remak--Schmidt--Azumaya. It follows that $X=\bigoplus_{\alpha} X_\alpha$, where $\End_\C(X_\alpha)$ is local. If we can show that $\End_{\C/\I}(X_\alpha)$ is local, we are done.

We know that $\End_{\C/\I}(X_{\alpha})=\End_\C(X_{\alpha})/\I(X_{\alpha},X_{\alpha})$. By property 1 of \Cref{def:admissible}, the identity on $X_{\alpha}$ cannot be an element of $\I(X_{\alpha},X_{\alpha})$, so $\End_{\C/\I}(X_{\alpha})$ is a nonzero quotient ring of a local ring, which is local by the ideal correspondence theorem for quotient rings.
\end{proof}

\begin{proposition}\label{prop:radical commutes with admissible}
    Let $\C$ be a category such that all endomorphism rings are artinian and $\I$ an admissible ideal.
    Then \[\Rad(\C / \I)=\Rad(\C) / \I.\] 
\end{proposition}
\begin{proof}
The equation $\Rad(\C / \I) = \Rad(\C) / \I$ holds if and only if $\Rad_{\C / \I}(X,Y) = \Rad_\C(X,Y) / \I$ holds for all pairs of objects $X,Y\in \C$.

First note that if $\End_{\C }(Y)$ is artinian, then
\begin{align*}
	 \Rad_\C(Y,Y) / \I &= \J(\End_{\C }(Y))/ \I=  \J(\End_{\C }(Y)+\I)/\I \\ 
	 &= \J(\End_{\C / \I}(Y)) =\Rad_{\C / \I}(Y,Y).
\end{align*}
Now we can see that 
\begin{align*}
	f \in \Rad_\C(X,Y) / \I &\Leftrightarrow f = f'+\I(X,Y)\text{, with } f'\in \Rad_\C(X,Y) \\
	& \Leftrightarrow  f = f'+\I(X,Y)\text{, s.t.\ } f'\circ g'\in \J(\End_{\C }(Y))\, \forall\, g'\in \Hom_\C  (Y,X)\\
	& \Leftrightarrow f\circ (g'+\I(Y,X))\in \J(\End_{\C }(Y)  )/\I\, \forall \, g'\in \Hom_\C  (Y,X)\\
	& \Leftrightarrow f\circ g \in \J(\End_{\C /\I}(Y)  )\, \forall\, g\in \Hom_{\C/\I}  (Y,X)\\
	& \Leftrightarrow f \in \Rad_{\C/\I}(X,Y). 
\end{align*}
\end{proof}
In particular, if $\C$ is $\kk$-linear and Krull--Remak--Schmidt, then all endomorphism rings are artinian and the above Proposition holds.

The following lemma is useful for our examples in \Cref{sec:examples}.
\begin{lemma}\label{lem:stack admissible}
    Let $\C$ be a small, $\Bbbk$-linear, Krull--Remak--Schmidt--Azumaya category and $\I$ an admissible ideal in $\C$.
    Let $\Jay$ be an admissible ideal in $\C/\I$ and $\tJay$ the set of morphisms $f$ in $\C$ such that $f+\I\in\Jay$ in $\C/\I$. Then the following hold:
    \begin{enumerate}
        \item $\tJay$ contains $\I$.
        \item $\tJay$ is an ideal.
        \item $\tJay$ is admissible in $\C$.
        \item $(\C/\I)/\Jay \simeq \C/\tJay$.
    \end{enumerate} 
\end{lemma}
\begin{proof}
    \textbf{1.} Let $f\in\I$.
    Then $f\mapsto 0$ in $\C/\I$.
    Since all zero morphisms in $\C/\I$ are in $\Jay$, we see $f\in \tJay$.
    
    \textbf{2.} Let $f:X\to Y$ be nonzero in $\tJay$ and let $g:Y\to Z$ be nonzero in $\C$.
    Then $f+\I$ and $g+\I$ are in $\Mor(\C/\I)$.
    So, $(g+\I)\circ(f+\I)$ is in $\Jay$ and is equal to $gf + \I$.
    Then $gf \in \tJay$.
    
    \textbf{3.} Let $f\in\tJay$.
    Then $f+\I\in\Jay$ and, by assumption, there exists $(g_1+\I),\ldots,(g_n+\I)$ in $\Mor(\C/\I)\setminus\Jay$ such that \[ f + \I = (g_n+\I)\circ (g_{n-1}+\I)\circ\cdots\circ (g_2+\I)\circ(g_1+\I).\]
    Then for each $g_i$ there is $h_i\in\I$ such that $g_i+h_i\mapsto g_i+\I$ and \[f = (g_n+h_n)\circ(g_{n-1}+h_{n-1})\circ\cdots\circ(g_2+h_2)\circ (g_1+h_1). \]
    Since each $g_i\notin \Jay$, we know each $g_i+h_i\notin\tJay$ and so $f$ is a finite composition of morphisms not in $\tJay$.
    Additionally, for any nonzero, nonisomorphism endomorphism $f$, we have $f^n\in\I$ for some $n\in\N$.
    Then $f^n\in\tJay$ by statement 1.
    Therefore, $\tJay$ is admissible.
    
    \textbf{4.} Recall $\Ob((\C/\I)/\Jay) = \Ob(\C/\tJay)$.
    We now produce a bijection between $\Mor((\C/\I)/\Jay)$ and  $\Mor(\C/\tJay)$ by producing bijections \[ \phi_{X,Y}:\Hom_{\C/\tJay}(X,Y) \to \Hom_{(\C/\I)/\Jay}(X,Y) \] for each ordered pair $X,Y$ of objects.
    
    Let $f+\tJay\in\Hom_{\C/\tJay}(X,Y)$.
    Then there exists $g\in\tJay\subset \Mor(\C)$ such that $f+g\mapsto f+\tJay\in\Mor(\C/\tJay)$.
    If $g\in\I$ then $f+g\mapsto f+\I\in\Mor(\C/\I)$; otherewise $f+g\mapsto f+g+\I\in\Mor(\C/\I)$.
    In either case, $f+g\mapsto f+\Jay$ in $\Hom_{(\C/\I)/\J}(X,Y)$.
    We define $\phi_{X,Y}(f+\tJay):= f+\Jay$.
    
    It is immediate that $\phi_{X,Y}$ is injective.
    Suppose $f+\Jay\in\Hom_{(\C/\I)/\Jay}$.
    Then there exists $g+\I$ in $\Hom_{\C/\I}(X,Y)$ such that $f+g+\I\mapsto f+\Jay$.
    Then there exists $h\in\Hom_{\C}(X,Y)$ such that $f+g+h\mapsto f+g+\I$.
    But this means $g+h\in\tJay$ and so $f+(g+h)\mapsto f+\tJay$ in $\Hom_{\C/\tJay}(X,Y)$.
    Thus, $\phi_{X,Y}$ is surjective.
\end{proof}

\section{Relations}\label{sec:relations}

In this section we look at two types of admissible ideals: those generated by point relations (\Cref{subsec:point relations}) and those generated by length relations (\Cref{subsec:length}).
These generalize a relation generated by a single path of length two and relations generated by all paths of a particular length, respectively.

\subsection{Point Relations}\label{subsec:point relations}
Here we generalize relations generated by a single path of length 2 to a point relation (\Cref{def:point relation}) and prove this generates an admissible ideal (\Cref{thm:point relations are admissible}).
We then give examples that point to a continuous version of a gentle algebra (\Cref{ex:monomial points,ex:crossing real lines}).

\begin{definition}\label{def:decomposition point}
    Let $f:X\to Y$ be a nonisomorphism between indecomposables in $\C$.
    A \emph{decomposition point} $Z$ in $\C$ is an indecomposable object such that there exists nonisomorphisms $g:X\to Z$ and $h:Z\to Y$ where $f=h\circ g$.
\end{definition}

\begin{definition}\label{def:acyclic morphism}
    Let $f:X\to Y$ be a nonisomorphism of indecomposables in $\C$, $Z$ a decomposition point of $f$, and $f=g\circ h$ such a decomposition.
    We call $f$ an \emph{acyclic morphism} if for all pairs $g':X\to Z$ and $h':Z\to Y$ such that $f=h'\circ g'$ then $h'$ and $g'$ are scalar multiples of $h$ and $g$, respectively.
\end{definition}

Note that an acyclic morphism cannot be irreducible. (I.e., it must be a path of length at least 2 in a quiver.)

\begin{definition}\label{def:point relation}
    Let $f$ be an acyclic morphism and $Z$ a decomposition point of $f$.
    Let $P$ be the set of all nonisomorphisms $g$ between indecomposables satisfying the following.
    \begin{itemize}
        \item There exists $h_1$ and $h_2$ morphisms of indecomposables such that $f=h_2\circ g\circ h_1$.
        \item We have $Z$ as a decomposition point of $g$.
    \end{itemize}
    Let $\mathcal P_{f,Z}$ be the ideal in $\C$ generated by $P$.
    We call $\mathcal P_{f,Z}$ the \emph{point relation through $Z$ by $f$}.
\end{definition}

\begin{definition}\label{def:admissible point relations}
    Let $\{\mathcal P_\alpha\}$ be a collection of point relations in $\C$.
    We say $\{\mathcal P_\alpha\}$ is \emph{admissible} if each morphism of indecomposables appears in at most finitely-many $\mathcal P_\alpha$.
\end{definition}

\begin{theorem}\label{thm:point relations are admissible}
    Let $\{\mathcal P_\alpha\}$ be an admissible collection of point relations in $\C$ and let $\I=\langle \bigcup_\alpha P_\alpha\rangle$.
    Suppose also that for each indecomposable $X$ in $\C$, we have $\End_{\C}(X) / \I(X,X)$ is finite dimensional.
    Then $\I$ is an admissible ideal.
\end{theorem}
\begin{proof}
We satisfy \Cref{def:admissible}(2) by assumption.

Now suppose $f\in \I$; we show that $f$ can be written as a finite composition of morphisms not in $\I$.
Since $f$ is a finite sum of morphisms of indecomposables, we assume without loss of generality $f$ is a morphism of indecomposables.
Then, by assumption there are at most finitely-many $\mathcal P_\alpha$ such that $f\in \mathcal P_\alpha$.

We proceed by induction beginning with $f$ is only in $\mathcal P_1$.
Let $f_1$, $P_1$, and $Z_1$ be as in \Cref{def:point relation}.
Then there exists morphisms $h_1$ and $h_2$ in $Mor(\C)$ and $g\in P_1$ such that $f = h_2\circ g \circ h_1$.
Further, $g=h'_2\circ h'_1$ where the target of $h'_1$ is $Z_\alpha$ and the source of $h'_2$ is $Z_\alpha$.
So, let $g_1 = h'_1\circ h_1$ and let $g_2=h_2\circ h'_2$.
Note that neither $g_1$ nor $g_2$ is in $\mathcal P_1$.
Further, neither $g_1$ nor $g_2$ is in $\I$ or else $f$ would be in another $\mathcal P_\alpha$ as well.
Thus, we have our desired decomposition.

Now assume that if $f$ is in $n$ of the $\mathcal P_\alpha$, then $f$ is a finite composition of morphisms not in $\I$.
Suppose $f$ is in $n+1$ of the $\mathcal P_\alpha$ and denote one of them by $\mathcal P_1$.
Let $f_1$, $P_1$, and $Z_1$ be as before for $\mathcal P_1$.
We find $g_1$ and $g_2$ as before, but they may be in $\I$.
However, each $g_1$ and $g_2$ may only be in $n$ or fewer $\mathcal P_\alpha$ and so are a finite composition of morphisms not in $\I$.
Therefore, $f$ is a finite composition of morphisms not in $\I$.
\end{proof}

\begin{example}[Discrete quiver]\label{ex:monomial points}
Let $Q$ be a discrete quiver. Then any quadratic monomial relation in $Q$ corresponds to a point relation in the $\kk$-linear categorification of $Q$.

In particular, any gentle algebra can be obtained by considering a quiver with point relations. 
\end{example}

\begin{example}[Continuous ``gentle'', crossing real lines]\label{ex:crossing real lines}
Consider two copies of the real line, labeled $\R$ and $\R'$. We label the numbers in $\R$ by $x$ and the numbers in $\R'$ by $x'$. Identify $0$ and $0'$ and label the category of $\kk$-representations  of the resulting partially ordered set by $\C$. 
\begin{figure}
    \centering
    \begin{tikzpicture}[inner sep = 0cm, outer sep = 0cm, xscale = 1, decoration={
    markings,
    mark=at position 0.5 with {\arrow{>}}}, very thick
    ]
    \coordinate (topleft) at (0,1);
    \coordinate (bottomleft) at (0,0);
    \coordinate (topright) at (10,1);
    \coordinate (bottomright) at (10,0);
    \coordinate (mid) at (5,0.5);
    \draw[blue, postaction=decorate] (topleft) ..controls +(4,0).. (mid); 
    \draw[red, postaction=decorate] (mid) ..controls +(1,-.5).. (bottomright);
    \draw[red, postaction=decorate] (bottomleft)  ..controls +(4,0).. (mid); 
    \draw[blue, postaction=decorate] (mid) ..controls +(1,.5).. (topright);
    \draw[fill=white] (mid) circle[radius=1mm];
    \end{tikzpicture}
    \caption{The category considered in \Cref{ex:crossing real lines}. The two copies of the real line have been drawn in different colours}
    \label{fig:my_label}
\end{figure}
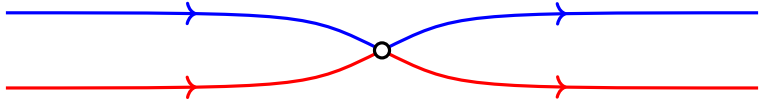

Let $\mathcal P$ be the point relation at $0$ generated morphisms starting in $\R_{<0}$ and ending in  $\R'_{>0}$. Dually, let $\mathcal P'$ be the point relation at $0$ generated morphisms starting in $\R'_{<0}$ and ending in  $\R_{>0}$. The collection $\{\mathcal P, \mathcal P'\}$ generates an admissible ideal.
In later work, we will argue that this $\C$ with this ideal yields a continuous generalization of a gentle algebra. 

\begin{remark}
If we do not assume that $\End_{\C}(X) / \I(X,X)$ is finite dimensional in our hypothesis of \Cref{thm:point relations are admissible}, it is possible that we do not have an admissible ideal.
See \Cref{ex:big wedge}.
\end{remark}

\end{example}
\subsection{Length Relations}\label{subsec:length}
We now generalize relations generated by all paths of a certain length to length relations.
To do this we define a way of measuring length in our category (\Cref{def:weakly archimedean monoid,def:category with length in Lambda}) and provide examples (\Cref{ex:weakly Archimedian monoids,ex:categories with length in Lambda}).
Then we define the length relations (\Cref{def:length relation}) and provide examples (\Cref{ex:length relations}) and prove that length relations generate admissible ideals (\Cref{thm:length relations are admissible}).
In \Cref{apx:length} we discuss the proof of \Cref{thm:length relations are admissible} (\Cref{subsec:need weakly archimedean}), why we require the specific setup that we have (\Cref{sec:more on length}), and compare our notion of length to the notion of a metric on a category, introduced by Lawvere \cite{L73} (\Cref{subsec:length vs metric}).

Recall a commutative monoid $\Lambda$ is a set with an associative, commutative, binary operation $+_{\Lambda}:\Lambda\times\Lambda \to \Lambda$ and an identity 0.
\begin{definition}\label{def:weakly archimedean monoid}
    Let $\Lambda$ be a commutative monoid.
    We say $\Lambda$ is \emph{weakly Archimedian} if it satisfies the following.
    \begin{itemize}
        \item There is a total order $\leq$ on $\Lambda$.
        \item If $\lambda\neq 0$ then $\lambda>0$.
        \item If $\lambda_1>\lambda_2$ then, for any $\lambda_3$, we have $\lambda_1+\lambda_3> \lambda_2+\lambda_3$ or $\lambda_1+\lambda_3 = \lambda_2+\lambda_3=\max \Lambda$.
        \item For all $0<\lambda_1<\lambda_2$ in $\Lambda$, there exists $n\in\N$ such that
        \begin{displaymath}
            n\lambda_1 := \underbrace{\lambda_1 +_{\Lambda} \lambda_1 +_{\Lambda} \cdots +_{\Lambda} \lambda_1}_{n} \geq \lambda_2.
        \end{displaymath}
    \end{itemize}
\end{definition}

\begin{example}\label{ex:weakly Archimedian monoids}
    We give three examples, two of which the reader might expect.
    \begin{enumerate}
        \item The set $\N$ with the usual total order and $+_{\N}$ given in the usual way is weakly Archimedian.
        \item The set $\R_{\geq 0}$ with the usual total order and $+_{\R}$ given in the usual way is weakly Archimedian.
        \item\label{ex:weakly Archimedian monoids:with max} Let $\Lambda = \{0,1,2,\ldots,n-1,n,\infty\}$. Let $+_{\Lambda}$ be given by
        \begin{displaymath}
            \lambda_1+_{\Lambda}\lambda_2 = \begin{cases}
                \lambda_1+_{\N}\lambda_2 & (\lambda_1+_{\N}\lambda_2)\leq n \\
                \infty & \text{otherwise}.
            \end{cases}
        \end{displaymath}
        For the total order, we say $0<1<2<\cdots<n-1<n<\infty$.
        Then $\Lambda$ is weakly Archimedian.
    \end{enumerate}
\end{example}
When the weakly Archimedian monoid \textcolor{karin}{structure} is clear, we write $+$ instead of $+_{\Lambda}$.

\begin{definition}\label{def:stem category}
    Let $\widehat{\C}$ be a small category and ignore any additional structure.
    The \emph{additive $\Bbbk$-linearization} $\C$ of $\widehat{\C}$ is the category such that
    \begin{itemize}
        \item $\Ob(\C)$ contains finite direct sums of objects in $\widehat{\C}$ and
        \item $\Hom_{\C}(X,Y) = \langle \Hom_{\widehat{\C}}(X,Y) \rangle_\Bbbk$ for indecomposables $X,Y\in\Ob(\C)$.
    \end{itemize}
    The other Hom spaces and composition are given by extending bilinearly and using composition in $\widehat{\C}$.
    
    Let $\C$ be a small $\Bbbk$-linear category.
    We say subcategory $\widehat{\C}$ of $\C$ is a \emph{stem category} of $\C$ if $\C$ is isomorphic to the additive $\Bbbk$-linearization of $\widehat{\C}$.
\end{definition}
This is similar to the use of ``stem category'' in \cite{Bongartz}.
We note that Bongartz considered stem categories of locally bounded categories.
We say $\C$ is \emph{locally bounded} if, for each object $X\in\Ob(\C)$, we have \[\dim_{\Bbbk}\left( \bigoplus_{Y\in\Ob(\C)} \Hom_{\C}(X,Y)\oplus\Hom_{\C}(Y,X) \right) <\infty. \]
Immediately, we note that when we consider the $\Bbbk$-linearization of $\R$ as a category, this condition fails.
In fact, many categories in the present paper are not locally bounded, especially those considered in \Cref{sec:examples}.

\begin{remark}\label{rmk:quiver additive linearization}
    If we consider a quiver $Q$ as a category $\widehat{\Q}$, then the additive $\Bbbk$-linearization of $\widehat{\Q}$ is precisely the $\Bbbk$-linear categorification $\Q$ of $Q$.
\end{remark}

\begin{definition}\label{def:category with length in Lambda}
    Let $\C$ be a $\Bbbk$-linear category, $\widehat{\C}$ be a stem category of $\C$, and $\Lambda$ a weakly archimedian monoid.
    
    We say \emph{$\C$ has length in $\Lambda$} if there is a function $\ell:\Mor(\widehat{\C})\to\Lambda$ satisfying the following.
    \begin{enumerate}
        \item If $f\in\Mor(\widehat{\C})$ is an isomorphism then $\ell(f)=0$.
        \item For each $f,g\in\Mor(\widehat{\C})$ such that $g\circ f\in\Mor(\widehat{\C})$, we have $\ell(g\circ f)=\ell(g)+\ell(f)$.
        \item For each $f\in\Mor(\widehat{\C})$ with $\ell(f)>0$ and for each $\lambda<\ell(f)$ there are $g,h\in\Mor(\widehat{\C})$ such that $f=h\circ g$ and either $\ell(g)=\lambda$ or $\ell(h)=\lambda$.
    \end{enumerate}
\end{definition}

\begin{example}\label{ex:categories with length in Lambda}
    We give some existing examples of \Cref{def:category with length in Lambda}.
    \begin{enumerate}
        \item\label{ex:categories with length in Lambda:classic} Let $Q$ be a quiver, $\Q$ the $\Bbbk$-linear categorification of $Q$, and $\widehat{\Q}$ a stem category of $\Q$ whose morphisms are generated by arrows.
        Let $\Lambda=\N$ and set $\ell(\alpha)=1$ for each morphism in $\widehat{\Q}$ from an arrow $\alpha$ in $Q$.
        Then $\Q$ has length in $\N$.
        \item\label{ex:categories with length in Lambda:AR} Let $\Q$ be the additive $\Bbbk$-linearization of a continuous quiver $\widehat{\Q}$ of type $A$ as in \cite{IRT}.
        Then $\widehat{\Q}$ is a stem category of $\Q$.
        
        Define $\ell: \Mor(\widehat{\Q})\to\R_{\geq 0}$ by $\ell(g_{x,y}) = |x-y|$ where $g_{x,y}$ is the unique nonzero morphism in $\widehat{\Q}$ from $x$ to $y$.
        Then $\Q$ has length in $\R_{\geq 0}$.
        Note $\Q$ is also a category with a metric.
        See \Cref{subsec:length vs metric}.
    \end{enumerate}
\end{example}

We now define a length relation.
\begin{definition}\label{def:length relation}
    Let $\Lambda$ be a weakly Archimedian monoid and $\C$ a category with length in $\Lambda$ with stem category $\widehat{\C}$.
    Consider $\Lambda_1,\Lambda_2$ subsets of $\Lambda$ such that $\Lambda_1\amalg\Lambda_2=\Lambda$, $|\Lambda_1|\geq 2$, and for all $\lambda_1\in\Lambda_1, \lambda_2\in\Lambda_2$ we have $\lambda_1<\lambda_2$.
    Then the set $\ell^{-1}(\Lambda_2)$ in $\widehat{\C}$ generates an ideal $\I$ in $\C$.
    We call $\I$ a \emph{length relation}.
\end{definition}

\begin{remark}\label{rmk:length is not a number}
    It is possible that $\Lambda_1$ has no maximum element and $\Lambda_2$ has no minimum element.
    (Consider, for example, $\Lambda=\mathbb{Q}_{\geq 0}$.)
    Thus, we may not always be able to say that we are taking ``paths longer than $\lambda$'' for some $\lambda\in\Lambda$.
\end{remark}

\begin{example}\label{ex:length relations}
    We give three examples of length relations.
    \begin{enumerate}
        \item Let $Q$ be a quiver and $\Q$ its $\Bbbk$-linear categorification, which has length in $\N$ (\Cref{ex:categories with length in Lambda}(\ref{ex:categories with length in Lambda:classic})).
        Let $\widehat{\Q}$ be the stem category of $\Q$ seen, effectively, as $Q$ embedded in $\Q$.
        Let $\Lambda_1 = \{0,1,2\}$ and $\Lambda_2=\{3,4,5,\cdots\}$.
        Then $\I=\langle \ell^{-1}(\Lambda_2)\rangle $ is the set of morphisms in $\Q$ generated by paths with length $\geq 3$ in $Q$.
        \item Any Nakayama algebra where the relations have constant length $l$ can be realized as the $\Bbbk$-linear categorification of its underlying quiver with length relations of length $l$ in $\N$.
        \item\label{ex:length relations:continuous A} Let $\Q$ and $\widehat{\Q}$ be as in \Cref{ex:categories with length in Lambda}(\ref{ex:categories with length in Lambda:AR}).
        Recall $\Q$ has length in $\R_{\geq 0}$.
        Let $\Lambda_1 =[0,4]$ and $\Lambda_2=(4,+\infty)$.
        Then $\langle\ell^{-1}(\Lambda_2)\rangle$ is the set of morphisms in $\Q$ of length \emph{strictly greater than $4$}.
    \end{enumerate}
\end{example}

\begin{theorem}\label{thm:length relations are admissible}
    Let $\Lambda$ be a weakly Archimedian monoid, $\C$ a category with length in $\Lambda$ with stem category $\widehat{\C}$, and $\I$ a length relation.
    If $\End_{\widehat{\C}}(X)$ is a finitely-generated monoid, for each $X\in\Ob(\widehat{\C})$, then $\I$ is an admissible ideal.
\end{theorem}
\begin{proof}
    If $\I=\emptyset$ then condition (1) is vacously satisfied.
    Assume $\I\neq\emptyset$, let $f\in \I$ such that $f\in\Mor(\widehat{\C})$, and let $\lambda\in\Lambda_1$ such that $\lambda>0$.
    Then there is $n\in\N$ such that $n\lambda \geq \ell(f)$.
    Thus, there is some decomposition $f=g_n\circ\cdots\circ g_1$ where $\ell(g_i)\in\Lambda_1$ for each $g_i$.
    Thus, each $g_i$ is not in $\I$.
    
    Since $\End_{\widehat{\C}}(X)$ is a finitely-generated monoid, let $m$ be the number of generators and let $\{f_i\}_{i=1}^m$ be the set of generators.
    Let
    \[ N = \max_i \{ \min_n \{ n\ell(f_i) \mid n\ell(f_i)\in\Lambda_2\}\}.\]
    Then
    \[ \dim_{\Bbbk}(\End_{\C}(X) / \I(X,X)) \leq m\cdot N + 1,\]
    where we need the ``$+1$'' to account for the identity in $\End_{\widehat{\C}}(X)$.
    Therefore, $\I$ is an admissible ideal.
\end{proof}

\section{Examples}\label{sec:examples}
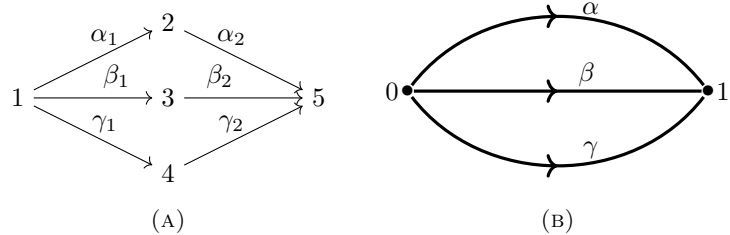
\begin{figure}[b]
    \centering
    \begin{subfigure}[b]{0.4\textwidth}
        \centering
        
\begin{tikzpicture}[xscale=2]
\node (1) at (0,0) {1};
\node (2) at (1,1) {2};
\node (3) at (1,0) {3};
\node (4) at (1,-1) {4};
\node (5) at (2,0) {5};

\draw[->] (1) -- node[pos=0.6, above]{$\alpha_1$} (2);
\draw[->] (1) -- node[pos=0.7, above]{$\beta_1$} (3);
\draw[->] (1) -- node[pos=0.6, above]{$\gamma_1$} (4);
\draw[->] (2) -- node[pos=0.4, above]{$\alpha_2$} (5);
\draw[->] (3) -- node[pos=0.3, above]{$\beta_2$} (5);
\draw[->] (4) -- node[pos=0.4, above]{$\gamma_2$} (5);
\end{tikzpicture}
        \caption{}
        \label{fig:three-path-discrete}
    \end{subfigure}
    \begin{subfigure}[b]{0.4\textwidth}
        \centering
        \begin{tikzpicture}[xscale=2, inner sep = 0cm, outer sep = 0cm, very thick,decoration={
    markings,
    mark=at position 0.5 with {\arrow{>}}}
    ]
\node (start) at (0,0) {$\bullet$};
\node (end) at (2,0) {$\bullet$};
\node at (-.1,0){$0$};
\node at (2.1,0){$1$};

\draw[postaction={decorate}] (start) .. controls (.5,1.3) and (1.5,1.3).. node[above=2pt, pos=0.6]{$\alpha$} (end);
\draw[postaction={decorate}] (start) -- node[pos=0.6, above=2pt]{$\beta$} (end);
\draw[postaction={decorate}] (start) .. controls (.5,-1.3) and (1.5,-1.3).. node[pos=0.6, above=2pt]{$\gamma$} (end);
\end{tikzpicture}
        \caption{}
        \label{fig:three-path-continuous}
    \end{subfigure}
    \caption{The quivers considered in \Cref{ex: discrete 3 paths,ex: cts 3 paths}, respectively.}
    \label{fig:three-path-quivers}
\end{figure}

\begin{example}\label{ex: discrete 3 paths}
Consider the (discrete) quiver $Q$ shown in \Cref{fig:three-path-discrete}.
The relation $\alpha_2\alpha_1-2\beta_2\beta_1+3\gamma_2\gamma_1$ generates an admissible ideal in the classical sense. In the $\kk$-linear categorification of $Q$, we rewrite the relation as the composition of morphisms
\[1\xrightarrow{\left [ \begin{smallmatrix}\alpha_1 \\ -2\beta_1\\3\gamma_1\end{smallmatrix}\right]} 2\oplus 3\oplus 4\xrightarrow{\left [ \begin{smallmatrix}\alpha_2 & \beta_2 &\gamma_2\end{smallmatrix}\right]}  5,
    \]
    which generates an admissible ideal
    
\end{example}

\begin{example} \label{ex: cts 3 paths}

Consider the continuous analogue of \Cref{ex: discrete 3 paths}, displayed in \Cref{fig:three-path-continuous}.
We consider a similar relation $\alpha-2\beta+3\gamma$.
Let $X$, $Y$, and $Z$ be points on the interior of the paths $\alpha$, $\beta$, and $\gamma$, respectively.
Then $\alpha=\alpha_2\alpha_1$ where $\alpha_1:0\to X$ and $\alpha_2:X\to 1$.
We similarly write $\beta=\beta_2\beta_1$ and $\gamma=\gamma_2\gamma_1$.
Then the relation $\alpha-2\beta+3\gamma$ can be written as the composition
\[0
\xrightarrow{\left [ \begin{smallmatrix}\alpha_1 \\ -2\beta_1\\3\gamma_1\end{smallmatrix}\right]}
X\oplus Y\oplus Z
\xrightarrow{\left [ \begin{smallmatrix}\alpha_2 & \beta_2 &\gamma_2\end{smallmatrix}\right]}
1,
    \]
and it generates an admissible ideal.
\end{example}

\begin{example}[Real line with point relations on integers]\label{ex:real line integers}
Let $\Q$ be the additive $\Bbbk$-linearization of $\R$ as a category where paths move upwards.
For a point $c\in \R$, let the (unique) point relation at $c$ be $\mathcal P_c$. The collection of point relations on the integers, $\{\mathcal P_z\}_{z\in \mathbb{Z}}$ generates an admissible ideal by \Cref{thm:point relations are admissible}.

The ``Auslander--Reiten space'' of the representations of this quiver is shaped like a mountain range; it is a set of triangles joined at their bottom vertices, see \Cref{fig:AR mountain range}
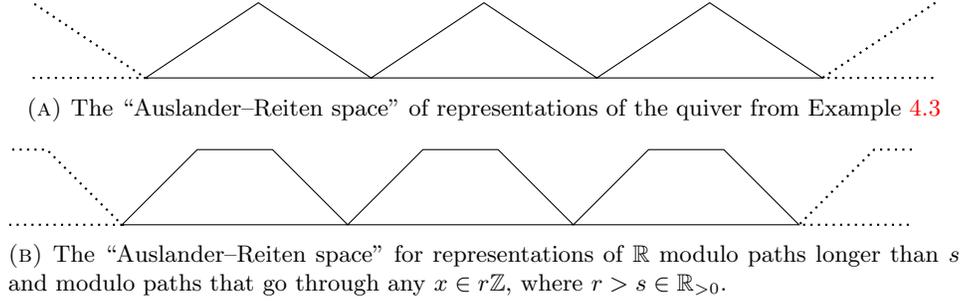
\begin{figure}
    \centering
    \begin{subfigure}{\textwidth}
    \centering
    \begin{tikzpicture}
        \draw[dotted, thick] (-1.5,0) -- (0,0)--(-1.5,1);
        \draw[dotted,thick] (10.5,1) -- (9,0) -- (10.5,0);
        \draw[] (0,0) -- (1.5,1)--(3,0)--(4.5,1)--(6,0)--(7.5,1)--(9,0)--(0,0);
    \end{tikzpicture}
    \caption{The ``Auslander--Reiten space'' of representations of the quiver from \Cref{ex:real line integers}}
    \label{fig:AR mountain range}
    \end{subfigure}
    
    \begin{subfigure}{\textwidth}
    \begin{tikzpicture}
    \draw[color=white] (0,0)--(0,1.3);
        \draw[dotted, thick] (-1.5,0) -- (0,0)--(-1,1)--(-1.5,1);
        \draw[dotted,thick] (10.5,1) -- (10,1) -- (9,0) -- (10.5,0);
        \draw[] (0,0) -- (1,1)--(2,1)--(3,0)--(4,1)--(5,1)--(6,0)--(7,1)--(8,1)--(9,0)--(0,0);
    \end{tikzpicture}
    
    \caption{The ``Auslander--Reiten space'' for representations of $\R$ modulo paths longer than $s$ and modulo paths that go through any $x\in r\Z$, where $r>s\in\R_{>0}$.}\label{fig:AR chopped mountains}
    \end{subfigure}
    \caption{The ``Auslander--Reiten spaces'' of representations of the quivers in \Cref{ex:real line integers,ex:real line points}}
    \label{fig:AR spaces}
\end{figure}
\end{example}

\begin{example}[Circle with length/Kupisch relations]\label{ex:circle length}
    Let $\Q$ be the additive $\Bbbk$-linearization of a continuous quiver $\widehat{\Q}$ of type $\widetilde{A}$ as in \cite{HR}.
    Define $\ell: \Mor(\widehat{\Q})\to\R_{\geq 0}$ by $\ell(f)=\phi-\theta + 2n\pi$ where $f:e^{i\theta}\to e^{i\phi}$, and $0\leq \phi-\theta<2\pi$, and $n$ is the number of full rotations around the circle at $e^{i\theta}$ before moving to $e^{i\phi}$.
    Then $\Q$ has length in $\R_{\geq 0}$.
    If $\Q$ is acyclic, we may replace $\R_{\geq 0}$ with $\Lambda=[0,2\pi)\cup\{\infty\}$ and define $+_\Lambda$ similarly to \Cref{ex:weakly Archimedian monoids}(\ref{ex:weakly Archimedian monoids:with max}).
    
    \begin{figure}
        \centering
        \begin{tikzpicture}[very thick, decoration={markings,mark=at position 0.5 with {\arrow{>}}}]
            \draw[radius=3cm] (0,0) circle;
            \draw[radius=2.8cm, color=red,postaction=decorate] (180:2.8cm) arc[start angle =180, delta angle= 180] node[above left, pos=0.666, red]{\XSolidBrush};
            \draw[radius=3.2cm, color=blue, postaction=decorate] (0:3.2cm) arc[start angle = 0, delta angle= 90] node[above right, pos=0.666, blue]{\CheckmarkBold};
            \foreach \x in {0, 30,...,360}
            	\draw[thin] (\x:2.9cm) -- (\x:3.1cm);  
        \end{tikzpicture}
        \caption{The circle with length relations as described in \Cref{ex:circle length}. In this figure, the relations have length $\frac{2\pi}{3}$.}
        
        \label{fig:circle length}
    \end{figure}
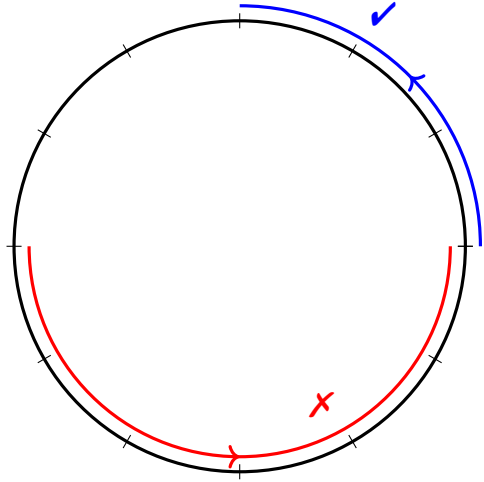
    
    Now assume $\widehat{\Q}$ has cyclic orientation.
    Let $\kappa$ be a Kupisch function as in \cite[Definition~3.9]{RZ}.
    That is, $\kappa:\R\to\R_{> 0}$ is a function such that $\kappa(t)+t>t$ and $\kappa(t+1)=\kappa(t)$, for all $t\in\R$.
    This yields a map $\mathbb{S}^1\to \R_{>0}$ where $\mathbb{S}^1=[0,1]/\{0\sim 1\}$.
    If $\kappa$ is constant with value $a$, then this yields a length relation where $\Lambda_1=[0,a]$ and $\Lambda_2=(a,+\infty)$.
    If $\kappa$ is not constant, then we do not have a length relation.
    However, if $\kappa$ does not have any separation points \cite[Definition~4.2]{RZ}, then $\kappa$ still induces an admissible ideal.
\end{example}

\begin{example}[Real line with length and point relations]\label{ex:real line points}
    Let $\Q$ be the additive $\Bbbk$-linearization of $\R$ as a category where paths move upwards.
    Let $r,s$ be positive real numbers and for each $x\in r\Z\subset \R$, let $\mathcal P_x$ be the (unique) point relation in $\Q$ through $x$ and $\I$ the admissible ideal generated by $\bigcup_{x\in r\Z} \mathcal{P}_x$.
    Let $\Jay$ be the the length relation in $\Q/\I$ obtained by modding out by paths of length greater than $s$.
    By \Cref{thm:length relations are admissible,thm:point relations are admissible} with \Cref{lem:stack admissible} we obtain an admissible ideal $\tJay$ given by the point relations at each $x\in r\Z$ and paths of length greater than $s$.
    
    If $r\leq s$ then $\C/\I = \C/\tJay$ since we cannot have a morphism of length greater than $r$ in $\C/\I$ anyway.
    If $r>s$ then we obtain paths of of length less than or equal to $s$ that do not pass through any $x\in r\Z$.
    The ``Auslander--Reiten space'' for the case $r>s$ is in \Cref{fig:AR chopped mountains}; notice the similarity with \Cref{fig:AR mountain range}.
\end{example}

\begin{example}[Complications with Cycles]\label{ex: cycles length}
    For each $n\in\N$, let $\C_n$ be circle whose radius is $\frac{1}{2}e^{-n}$.
    Let $\C$ be the additive $\kk$-linearization of $\R_{\leq 0}\amalg(\coprod_{n\in\N} \C_n)/ \sim$, where $\R\ni -n\sim 0\in\C_n$.
    See \Cref{fig:cycles with decreasing length} for a visual depiction.
    
    We see that $\C$ has length in $\R_{\geq 0}$.
    Let $\I$ be a length relation.
    Since our length is in $\R_{\geq 0}$ we can say we are modding out by length $>L$ or $\geq L$ for some $L>0\in\R$.
    
    Notice that for each $N\in\N$, there exists some $\C_n$ with radius $r$ such that $Nr < L$.
    Therefore, there is no natural number $n$ such that for all nonisomorphism endomorphisms $f$ we have $f^n\in\I$.
    
    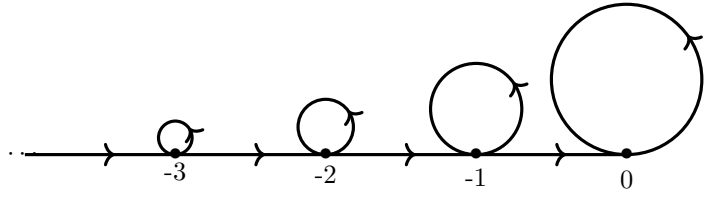
\begin{figure}
        \centering
        \begin{tikzpicture}[scale=2, decoration={markings, mark=at position 0.6 with {\arrow{>}}}]
    \foreach \x in {0, -1,...,-3}{
    	\node[outer sep=4+\x] (\x) at (\x,0){$\bullet$};
    	\node at (\x.south){\x};}
    \node (-4) at (-4,0){$\cdots$};
    \foreach \x [remember=\x as \y (initially 0)]in {-1, -2,...,-4}
    	\draw[very thick, postaction=decorate] (\x.center)--(\y.center);
    \foreach \x [evaluate=\x as \r using .5*e^(\x/2)] in {0, -1,...,-3}
         \draw[radius=-\r, very thick, postaction=decorate](\x)+(0,\r ) circle;
    \end{tikzpicture}
        \caption{Illustration of $\C$ in \Cref{ex: cycles length}, where we have glued smaller and smaller circles to each non-positive integer $n$ in $\R_{\leq 0}$.}
        \label{fig:cycles with decreasing length}
    \end{figure}
\end{example}

\begin{example}[Big wedge]\label{ex:big wedge}
    Let $\C$ be a cyclic continuous quiver of type $\widetilde{A}$ as in \cite{HR}.
    Let $\widehat{\Q} = (\coprod_{\N} \C )/\sim$ where we join all the copies of the $\C$ together at one point.
    Denote the wedge point by $X$.
    Let $\Q$ be the additive $\Bbbk$-linearization of $\widehat{\Q}$.
    Let us discuss the contruction of an admissible ideal out of point relations and a length relation.

    Notice $\Q$ has length in $\R_{\geq 0}$.
    However, since the endomorphism ring of $X$ is an infinitely-generated monoid in $\widehat{\Q}$, we see that $\End_{\C}(X)/ \I(X,X)$ is not finite-dimensional.
    Instead, we must add a point relation on all but finitely-many different copies of $\C$.
    If we do not have a point relation on all the cycles, we also need a length relation.
    Without such a combination, it is not possible to build an admissible ideal out of point relations and a length relation.
\end{example}

\subsection{The Real Plane}\label{subsec:real plane}
    We now consider a continuous version of the grid quiver with commutativity relations as examined in \cite{BBOS}.
    Let $\widehat{\Q}$ be the category whose objects are points in $\R^2$.
    
    We now define the $\Hom$ set between an arbitrary pair of points $(x,y)$ and $(z,w)$.
    Hom sets are given by considering paths made of up of horizontal and vertical line segments.
    For a pair $(x,y)$ and $(z,w)$, consider the set $P_{x,y}^{z,w}$ of all finite sequences $\{(x_i,y_i)\}_{i=i}^n$ such that
    \begin{itemize}
        \item $(x_1,y_1)=(x,y)$ and $(x_n,y_n)=(z,w)$,
        \item $x_1\leq x_2\leq\cdots \leq x_n$ and $y_1\leq y_2\leq\cdots \leq y_n$,
        \item $(x_i,y_i)\neq (x_{i+1},y_{i+1})$ for all $1\leq i<n$,
        \item $x_1=x_2$ or $y_1=y_2$,
        \item for all $1\leq i<n-1$, if $x_i=x_{i+1}$ then $y_{i+1}=y_{i+2}$, and
        \item for all $1\leq i<n-1$, if $y_i=y_{i+1}$ then $x_{i+1}=x_{i+2}$
    \end{itemize}
    We define \[\Hom_{\widehat{\Q}}((x,y),(z,w))=P_{x,y}^{z,w}.\]
    
    Note that $P_{x,y}^{z,w}$ may be empty.
    If either $x>z$ or $y>w$, then
    \[ \Hom_{\widehat{\Q}}((x,y), (z,w)) = \emptyset.\]
    If (i) $x=z$ and $y< w$ or (ii) $x< z$ and $y=w$, then
    \[ \Hom_{\widehat{\Q}}((x,y), (z,w)) = \left\{ \{ (x,y), (z,w) \}\right\}. \]
    If $(x,y)=(z,w)$ then
    \[ \Hom_{\widehat{\Q}}((x,y), (x,y)) = \left\{ \{ (x,y) \}\right\}. \]
    Composition is given by concatenating sequences and, if necessary, deleting a repeated term.
    
    Let $\Q$ be the additive $\Bbbk$-linearization of $\widehat{\Q}$ and let $(x,y),(z,w)\in \R^2$ such that $x<z$ and $y<w$.
    Then, we define
    \[
    \I((x,y),(z,w)) = \left\langle \{(x_i,y_i)\}_{i=1}^m - \{(x'_j,y'_j)\}_{j=1}^n \mid \{(x_i,y_i)\}_{i=1}^m \neq \{(x'_j,y'_j)\}_{j=1}^n \right\rangle.
    \]
    If $x=z$ or $y=w$, then $\I((x,y),(z,w))=0$.
    
    Consider
    \begin{align*}
        \{(x_i,y_i)\}_{i=1}^m - \{(x'_j,y'_j)\}_{j=1}^n &\in \I((x,y),(z,w)) \\
        \{(z_k,y_k)\}_{k=1}^p - \{(z'_l,z'_l)\}_{l=1}^q &\in \I((z,w),(u,v)).
    \end{align*}
    We show the composition is in $\I((x,y),(u,v))$.
    \begin{align*}
       & (\{(z_k,y_k)\}_{k=1}^p - \{(z'_l,z'_l)\}_{l=1}^q)\circ (\{(x_i,y_i)\}_{i=1}^m - \{(x'_j,y'_j)\}_{j=1}^n) \\
       =& \{(z_k,y_k)\}\circ \{(x_i,y_i)\} - \{(z_k,y_k)\}\circ \{(x'_j,y'_j)\} \\& - \{(z'_l,z'_l)\}\circ \{(x_i,y_i)\} + \{(z'_l,z'_l)\}\circ \{(x'_j,y'_j)\}.
    \end{align*}
    We know 
    \begin{align*}
        \{(z_k,y_k)\}\circ \{(x_i,y_i)\} & \neq \{(z_k,y_k)\}\circ \{(x'_j,y'_j)\} \\
        &\text{and} \\
        \{(z'_l,z'_l)\}\circ \{(x_i,y_i)\} & \neq \{(z'_l,z'_l)\}\circ \{(x'_j,y'_j)\}.
    \end{align*}
    Thus, the composition is in $\I((x,y),(u,v))$.
    
    Consider $\{(x_i,y_i)\}_{i=1}^m - \{(x'_j,y'_j)\}_{j=1}^n\in\I((x,y),(z,w))$.
    By assumption there is $1\leq k \leq m$ and $1\leq l \leq n$ such that $(x_k,y_k)\neq (x'_l,y'_l)$.
    Let
    \begin{align*}
        f:(x,y)\to (x_k,y_k)\oplus(x'_l,y'_l) &= \left[\begin{matrix} \{(x_i,y_i)\}_{i=1}^k \\ - \{(x'_j,y'_j)\}_{j=1}^l \end{matrix} \right] \\
        g:(x_k,y_k)\oplus(x'_l,y'_l) \to (z,w) &= \left[\begin{matrix} \{(x_i,y_i)\}_{i=k}^m & \{(x'_j,y'_j)\}_{j=l}^n \end{matrix} \right].
    \end{align*}
    Notice $f\notin \I((x,y),(x_k,y_k)\oplus(x'_l,y'_l))$, $g\notin \I((x_k,y_k)\oplus(x'_l,y'_l), (z,w))$, and
    \[ \{(x_i,y_i)\}_{i=1}^m - \{(x'_j,y'_j)\}_{j=1}^n = g\circ f.\]
    Thus, every morphism in $\I$ is given by a finite composition of morphisms not in $\I$.
    Furthermore, there are no cycles in $\Q$.
    Thus, $\I$ is an admissible ideal.
    
    The resulting $\Q/\I$ is the category where the objects are fnite direct sums of points in $\R^2$ and Hom spaces between points is given by
    \[
        \Hom_{\Q/\I}((x,y),(z,w)) = \begin{cases}
            \Bbbk & x\leq z \text{ and }y\leq w \\
            0 & \text{otherwise}.
        \end{cases}
    \]
    This means that $\Q/\I$ is the continuous generalization of (finite) discrete commutative grid quivers as examined in \cite{BBOS}.

\appendix

\section{More on Length Relations}\label{apx:length}

In this appendix we first discuss why the proof of \Cref{thm:length relations are admissible} fails if we do not require the weakly archimedian property or the finitely-generated monoid condition (\Cref{subsec:need weakly archimedean}).
Then we discuss why we require that every element $\lambda$ in a chosen $\Lambda$ must be the sum of finitely many elements $\lambda_1,\ldots,\lambda_n$ using an example (\Cref{sec:more on length}).
Finally, we compare  our definition of a category $\C$ having length in $\Lambda$ to the notion of a metric on the category $\C$, originally introduced by Lawvere \cite{L73} (\Cref{subsec:length vs metric}).

\subsection{Discussion of Proof of \Cref{thm:length relations are admissible}}\label{subsec:need weakly archimedean}
We first discuss the weakly Archimedian property and then discussion the assumption on the monoid $\End_{\widehat{\C}}(X)$.

Note the weakly Archimedian property is essential to the proof of \Cref{thm:length relations are admissible}.
If $\Lambda$ were simply a commutative monoid that respected some total order, we would not be guaranteed either of the two properties of an admissible ideal.
We present two examples that fail each of requirements (1) and (2) in \Cref{def:admissible} without failing the other.

First we consider an example where a morphism $f\in\I$ may not be a finite composition of morphisms not in $\I$.
Consider representations of $\R\cup\{+\infty\}$ with length of morphisms defined in the following way.
Suppose $\Lambda = \R_{\geq 0}\cup \{\infty\}$ where $a+_{\R}b$ is the standard addition when $a,b\in\R$ and $a+_{\R}b=\infty$ if $a=\infty$ or $b=\infty$.
The order on $\Lambda$ is the usual order on $\R$ and $x<\infty$ for all $x\in\R$.
Suppose $\Lambda_2=\{\infty\}$.
Then each morphism $f:+\infty\to x$ has $\ell(f)=\infty$ but there is no finite composition $g_n\circ\cdots\circ g_1=f$ where $\ell(g_i)<\infty$ for each $g_i$.
Thus $\ell^{-1}(\Lambda_2)$ fails requirement (1) but satisfies requirement (2).

Now we consider an example where all powers of loops are not in $\I$.
Let $Q$ be a quiver with loops and let $\Lambda=\N\cup\{\infty\}$.
Let $\Q$ be the categorification of $Q$ with length in $\Lambda$ by assigning $\ell(\alpha)=1$ for each arrow in $Q$.
Define $+$ to be the usual addition in $\N$ and $a+b=\infty$ if $a=\infty$ or $b=\infty$.
Further, the total order is the usual order on $\N$ and $n<\infty$ for all $n\in\N$.
Let $\Lambda_2=\{\infty\}$ and note that, for each loop $\alpha$ and for every $n\geq 1\in\N$, we have $\alpha^n\notin\ell^{-1}(\Lambda_2)$.
Thus, $\ell^{-1}(\Lambda_2)$ fails requirement (2) but satisfies requirement (1).

In \Cref{ex:big wedge}, we see that if we have an infinite wedge of circles, and our length relation is longer than the circumference of the circle, then the endomorphism ring of the wedge point in $\Q/\I$ cannot be finite-dimensional.
This is why we add the assumption that $\End_{\widehat{\C}}(X)$ is a finitely-generated monoid for all $X\in\Ob(\widehat{\C})$.

\subsection{More on the Definition of Length Relation}\label{sec:more on length}
Here we explain why we need the definition that we have.

Let $\Q$ be the semi-continuous quiver where we have the ascending real interval $[0,100]$ and an arrow $\alpha:-1\to 0$ (see \Cref{fig:semi-continuous with no length 1}).
If we want to put a length on $\Q$ we run into problems.
Notice the only possible length on a path $x\to y$, for $x,y\in[0,100]$, is a variation on using $y-x$.

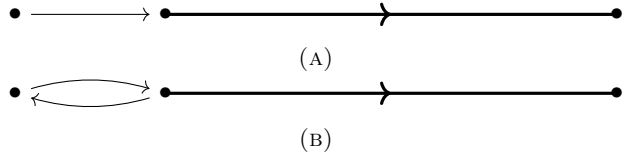
\begin{figure}
    \centering
    \begin{subfigure}[b]{\textwidth}
    \centering
    \begin{tikzpicture}[xscale=2, decoration={markings, mark=at position 0.5 with {\arrow{>}}}
    ]
    \node (l) at (0,0) {$\bullet$};
    \node (m) at (1,0) {$\bullet$};
    \node (r) at (4,0) {$\bullet$};
    \draw[->] (l) -- (m);
    \draw[very thick, postaction=decorate] (m.center) -- (r.center);
    \end{tikzpicture}
    \caption{}
    \label{fig:semi-continuous with no length 1}
    \end{subfigure}
    \\
\begin{subfigure}[b]{\textwidth}
    \centering
    \begin{tikzpicture}[xscale=2, decoration={
    markings,
    mark=at position 0.5 with {\arrow{>}}}
    ]
    \node (l) at (0,0) {$\bullet$};
    \node (m) at (1,0) {$\bullet$};
    \node (r) at (4,0) {$\bullet$};
    \draw[->] (l) to[bend left] (m);
    \draw[->] (m) to[bend left] (l);
    \draw[very thick, postaction=decorate] (m.center) -- (r.center);
    \end{tikzpicture}
    \caption{}
    \label{fig:semi-continuous with no length 2}
\end{subfigure}
    
    \caption{Two semi-continuous quivers $\Q,\Q^+$ ((a),(b), respectively) on which we cannot put a length.}
\end{figure}

We cannot give the discrete arrow $\alpha$ some infinitesimal length.
This would violate property (3) of \Cref{def:weakly archimedean monoid}, since there is no finite sum of infinitesimal lengths that would yield a real length.
We also cannot arbitrarily decide that two times the infinitesimal length (or any finite multiple) is equal to some real number, since our weakly archimedean monoid must be totally ordered and the addition must respect the order.

We cannot assign a real positive length to the arrow $\alpha$, either.
Suppose the length of the arrow is $d$.
Then, for any $\lambda<d \in\R_\geq 0$, we cannot satisfy (3) in \Cref{def:category with length in Lambda}.
If we ignored this problem, then we risk disconnecting our category if we mod out by a length less than $d$.

However, there is hope.
If we are careful, we can mod out by paths that \emph{feel} too long.
For example, one may check that the ideal generated by \[\{x\to y | x,y\in[0,100], y-x>3\} \cup \{(x\to y)\circ \alpha \mid x=0, y-x> 2\}\] is an admissible ideal which makes the arrow \emph{feel} as though it has length 1 when we want to mod out by paths whose lengths feel greater than 3.

Suppose we add an arrow $\beta:0\to -1$.
Call this new semi-continuous quiver $\Q^+$ (see \Cref{fig:semi-continuous with no length 2}).
Then we may produce a new admissible ideal by adding to our previous generating set:
\begin{align*}
    &\{x\to y | x,y\in[0,100], y-x>3\} \cup \{(x\to y)\circ \alpha \mid x=0, y-x> 2\} \\ & \cup \{(x\to y)\circ\alpha\beta \mid x=0, y-x> 1\} \cup \{(x\to y)\circ \alpha\beta\alpha \mid x=0, y\neq x\} \\ &\cup \{\alpha\beta\alpha\beta,\beta\alpha\beta\alpha\}.
\end{align*}
This again gives the feeling that the arrows have length 1 and feels like we're modding out by paths of length greater than 3 without violating our definitions.

\subsection{Metric versus Length}\label{subsec:length vs metric}
    In \cite{L73}, Lawvere introduced metrics on categories by formalising the relationship between the triangle inequality $\dist(b,c)+\dist(a,b) \geq \dist(a,c)$, where $\dist$ is some metric, and morphism composition $\Hom_\C(B,C)\otimes\Hom_\C(A,B)\to \Hom_\C(A,C)$, where $\C$ is some category.
    Lawvere allows for some non-symmetric constructions.
    However, we note that this is \emph{not} where our definition and Lawvere's definition diverge.
    
    We show that our example $\mathcal{Q}^+$ from \Cref{subsec:need weakly archimedean} has a metric that does not have a compatible length function (\Cref{ex:yes metric no length}).
    Then we show there is a category $\mathcal{X}$ with length $\ell$ in $\mathbb{Q}$ that does not have a compatible metric (\Cref{ex:no metric yes length}).
    
    \begin{example}\label{ex:yes metric no length}
    Consider the example $\Q^+$ in \Cref{subsec:need weakly archimedean}.
    There is a canonical (symmetric) metric to put on $\Q^+$ given by setting $\dist(x,y)$ to $|x-y|$ in $\R$ and then taking the different using the Euclidean metric of $\R^n$ for direct sums of points.
    However, as we have noted, this is not a category with length in $\R$.
    \end{example}
    
    \begin{example}\label{ex:no metric yes length}
    Let $\mathcal{X}'$ be the category obtained by considering disjoint closed intervals in $\mathbb{Q}$ of the following infinite sequences of lengths:
    \begin{align*}
        \{x_i\}_0^\infty &= \{3, 3.1, 3.14, 3.141, 3.1415, 3.14159, \ldots\} \\
        \{y_i\}_0^\infty &= \{3, 2.9, 2.86, 2.859, 2.8585, 2.85841, \ldots\}
    \end{align*}
    Specifically, we take disjoint copies of the intervals $[0,x_i]$ and $[0,y_i]$ in $\mathbb{Q}$, for each $x_i$ and $y_i$.
    Notice each $y_i=6-x_i$.
    Furthermore, in $\R$, the limit of $\{x_i\}$ is $\pi$ and the limit of $\{y_i\}$ is $6-\pi$.
    Thus, neither the limit of $\{x_i\}$ nor the limit of $\{y_i\}$ exists in $\mathbb{Q}$.
    
    Now, let $\widehat{\mathcal{X}}$ be the result of
    \begin{itemize}
        \item identifying all $0\in [0,x_i]$ and $0\in[0,y_i]$ together and
        \item identifying all $x_i\in[0,x_i]$ and all $y_i\in[0,y_i]$ together.
    \end{itemize}
    By $0$ we denote the equivalence class of all $0$'s and by $X$ we denote the equivalence class of all $x_i$'s and $y_i$'s Notice $\Hom_{\mathcal{X}}(0,X):=\N\times\{0,1\}$.
    For any other $x,y\in\Ob(\widehat{\mathcal{X}}$, $\Hom_{\mathcal{X}}(x,y):=\Hom_{\mathcal{X}'}(x,y)$.
    Let $\mathcal{X}$ be the additive $\Bbbk$-linearization of $\widehat{\mathcal{X}}$.
    
    Take $\Lambda=\mathbb{Q}$.
    For any pair $(x,y)$ in $\mathcal{X}$ such that $\Hom_{\mathcal{X}}(x,y)\neq 0$ and $(x,y)\neq (0,X)$, define $\ell(f)=|x-y|$ for each $f\in\Hom_{\mathcal{X}}(x,y)$.
    Let $f\in\Hom_{\mathcal{X}}(0,X)$ be nonzero such that $f$ is in copy $(i,j)$ of $\Bbbk$, where $(i,j)\in\N\times{0,1}$.
    Then set \[\ell(f)=\begin{cases} x_i & j=0 \\ y_i & j=1. \end{cases}\]
    We see this generates $\ell$ such that $\mathcal{X}$ has length in $\mathbb{Q}$.
    
    Let $L$ be the set of lengths of paths from $0$ to $X$ in $\mathcal{X}$.
    We see that, in $\mathbb{Q}$, the set $L$ has no infimum (or supremum) so we cannot use this value to define the distance between $0$ and $X$.
    If we define the length between $0$ and $X$ to be greater than $6-\pi$, we see there are infinitely-many paths that cause the triangle inequality to fail.
    Anything less than $6-\pi$ yields a distance incompatible with the length of the morphisms from $0$ to $X$.
    Thus we would need to use the infimum, $6-\pi$, which does not exist in our $\Lambda$.
    Thus, we cannot create a metric $\dist$ on $\mathcal{X}$ that is compatible with the length function $\ell$.
    \end{example}
    
    It is important to note that our specific \Cref{ex:no metric yes length} can be ``fixed'' by taking lengths in $\R$ instead of $\mathbb{Q}$.
    However, the point is to show that if we take an arbitrary $\Lambda$ then there is no guarantee we can construct a compatible metric directly.


\begin{thebibliography}{0}

\bibitem[AH]{AH}
I.~Assem and D.~Happel,
{\it Generalized tilted algebras of type {$A_{n}$}}
Comm. Algebra 9 (1981), no. 20, 2101–2125. \href{https://doi.org/10.1080/00927878108822697}{https://doi.org/10.1080/00927878108822697}.

\bibitem[ASS]{ASS}
I.~Assem, D.~Simson and A.~Skowro\'{n}ski, 
{\it Elements of the representation theory of associative algebras. Vol. 1.
Techniques of representation theory.}
London Mathematical Society Student Texts, 65.
Cambridge University Press, Cambridge, 2006. x+458 pp. \href{https://doi.org/10.1017/CBO9780511614309}{https://doi.org/10.1017/CBO9780511614309}.

\bibitem[AS]{AS}
I.~Assem and A.~Skowro\'{n}ski,
{\it Iterated tilted algebras of type {$\tilde{\bf A}_n$}}
Math. Z. 195 (1987), no. 2, 269–290. \href{https://doi.org/10.1007/BF01166463}{https://doi.org/10.1007/BF01166463}.

\bibitem[ARS]{ARS}
M.~Auslander, I.~Reiten and S.~Smal{\o},
{\it Representation theory of Artin algebras.}
Cambridge Studies in Advanced Mathematics, 36. 
Cambridge University Press, Cambridge, 1995. xiv+423 pp.
\href{https://doi.org/10.1017/CBO9780511623608}{https://doi.org/10.1017/CBO9780511623608}.

\bibitem[BBOS]{BBOS}
U.~Bauer, M.~B.~Botnan, S.~Oppermann, J.~Steen. {\it Cotorsion torsion triples and the representation theory of filtered hierarchical clustering}, Adv. Math. 369 (2020), 107171, 51 pp. 
\href{https://doi.org/10.1016/j.aim.2020.107171}{https://doi.org/10.1016/j.aim.2020.107171}.

\bibitem[BBH]{BBH}
B.~Blanchette, T.~Br\"ustle, and E.~J.~Hanson,
{\it Homological approximations in persistence theory}
arXiv:2112.07632v2 [math.AT] (2022),
\href{https://doi.org/10.48550/arXiv.2112.07632}{https://doi.org/10.48550/arXiv.2112.07632}.

\bibitem[Bo]{Bongartz}
K.~Bongartz, {\it On representation-finite algebras and beyond},  Advances in Representation Theory of Algebras, published by the European Mathematical Society, 2014, 65--101, \href{https://doi.org/10.4171/125-1/3}{https://doi.org/10.4171/125-1/3}.

\bibitem[CdSGO]{CdSGO}
F.~Chazal, V.~de Silva, M.~Glisse, and S.~Oudot, {\it The Structure and Stability of Persistence Modules}, Springer Briefs in Mathematics. Springer, Cham, \href{https://doi.org/10.1007/978-3-319-42545-0_2}{https://doi.org/10.1007/978-3-319-42545-0$\_$2}.

\bibitem[C-B]{C-B}
W.~Crawley-Boevey. {\it Functorial filtrations. II. Clans and the Gel‘fand problem}, J. London Math. Soc. (2) \textbf{40} (1989), no.~1, 9--30.
\href{https://doi.org/10.1112/jlms/s2-40.1.9}{https://doi.org/10.1112/jlms/s2-40.1.9}.

\bibitem[E]{Erdmann}
K.~Erdmann, {\it  Schur algebras of finite type} Quart. J. Math. Oxford Ser. (2) 44 (1993), no. 173, 17–41.
\href{https://doi.org/10.1093/qmath/44.1.17}{https://doi.org/10.1093/qmath/44.1.17}.

\bibitem[G]{Gabriel}
P.~Gabriel, {\it Unzerlegbare Darstellungen. I. } Manuscripta Math. 6 (1972), 71–103; correction, ibid. 6 (1972), 309.
\href{https://doi.org/10.1007/BF01298413}{https://doi.org/10.1007/BF01298413}.

\bibitem[HR]{HR}
E.~Hanson, J.~D.~Rock, {\it Decomposition of Pointwise Finite-Dimensional $\mathbb S^1$ Persistence Modules}, preprint, arXiv:2006.13793. \href{https://doi.org/10.48550/arXiv.2006.13793}{https://doi.org/10.48550/arXiv.2006.13793}.

\bibitem[IRT]{IRT}
K.~Igusa, J.~D.~Rock and G.~Todorov, {\it Continuous Quivers of Type A (I) Foundations}, Rendiconti del Circolo Matematico di Palermo Series 2 (2022). \href{https://doi.org/10.1007/s12215-021-00691-x}{https://doi.org/10.1007/s12215-021-00691-x}.

\bibitem[J]{Jasso}
G.~Jasso, {\it {$n$}-abelian and {$n$}-exact categories}, Math. Z. 283 (2016), no. 3-4, 703–759. \href{https://doi.org/10.1007/s00209-016-1619-8}{https://doi.org/10.1007/s00209-016-1619-8}

\bibitem[Kr]{Krause}
H.~Krause, {\it Krull--Schmidt categories and projective covers}, Expo. Math. 33 (2015), no. 4, 535–549.
\href{https://doi.org/10.1016/j.exmath.2015.10.001}{https://doi.org/10.1016/j.exmath.2015.10.001}

\bibitem[Ku]{Kupisch}
H.~Kupisch, {\it Symmetrische Algebren mit endlich vielen unzerlegbaren Darstellungen. I.}, J. Reine Angew. Math. (1965), no.~219, 1--25, \href{https://doi.org/10.1515/crll.1965.219.1}{https://doi.org/10.1515/crll.1965.219.1}.

\bibitem[Kv]{Kvamme}
S.~Kvamme, {\it Axiomatizing subcategories of {A}belian categories}, J. Pure Appl. Algebra 226 (2022), no. 4, Paper No. 106862, 27 pp. 
\href{https://doi.org/10.1016/j.jpaa.2021.106862}{https://doi.org/10.1016/j.jpaa.2021.106862}

\bibitem[L]{L73}
F.~W.~Lawvere, {\it Metric spaces, generalized logic, and closed categories}, Rend. Sem. Mat. Fis. Milano 43 (1974), 135--166.
\href{https://doi.org/10.1007/BF02924844}{https://doi.org/10.1007/BF02924844}

\bibitem[R]{Ringel}
C.~M.~Ringel, {\it The preprojective algebra of a quiver}, Algebras and modules, II (Geiranger, 1996), 467–480, 
CMS Conf. Proc., 24, Amer. Math. Soc., Providence, RI, 1998.  

\bibitem[RZ]{RZ}
J.~D.~Rock and S.~Zhu, {\it Continuous Nakayama Representations}, \href{https://doi.org/10.48550/arXiv.2207.03908}{https://doi.org/10.48550/arXiv.2207.03908}

\bibitem[SY]{SY}
A.~Skowro\'{n}ski, K.~Yamagata, {\it Frobenius algebras. I. }
Basic representation theory. EMS Textbooks in Mathematics. European Mathematical Society (EMS), Zürich, 2011. xii+650 pp. ISBN: 978-3-03719-102-6 
\href{https://doi.org/10.4171/102}{https://doi.org/10.4171/102}

\bibitem[S]{Street}
R.~Street, {\it Ideals, radicals, and structure of additive categories}, Appl. Categ. Structures 3 (1995), no. 2, 139–149. 
\href{https://doi.org/10.1007/BF00877633}{https://doi.org/10.1007/BF00877633}

\bibitem[V]{Vaso}
L.~Vaso, {\it $n$-cluster tilting subcategories of representation-directed
   algebras}, J. Pure Appl. Algebra 223 (2019), no. 5, 2101–2122.
\href{https://doi.org/10.1016/j.jpaa.2018.07.010}{https://doi.org/10.1016/j.jpaa.2018.07.010}

\end{thebibliography}
\end{document}